\documentclass[a4paper]{amsart}
\usepackage[latin1]{inputenc}
\usepackage{amssymb}
\usepackage{amsmath}
\usepackage{mathrsfs}
\usepackage{eufrak}
\usepackage{amsthm}
\usepackage{amsfonts}
\usepackage{textcomp}
\usepackage{graphicx}
\usepackage[pdftex]{color}
\usepackage{paralist}
\usepackage[shortlabels]{enumitem}
\usepackage{hyperref}
\usepackage{comment}
\usepackage[arrow, matrix, curve]{xy}
\def\mystrut{\scalebox{1.4}{\strut}}

\newtheorem*{corollary*}{Corollary}
\newtheorem{theorem}{Theorem}[section]
\newtheorem{theorem*}{Theorem}

\newtheorem{corollary}[theorem]{Corollary}
\newtheorem{lemma}[theorem]{Lemma}
\newtheorem{proposition}[theorem]{Proposition}

\newtheorem*{claim*}{Claim}
\newtheorem*{conjecture}{Conjecture}
\newtheorem*{Theorem A}{Theorem A}
\newtheorem*{Theorem B}{Theorem B}
\newtheorem*{Theorem C}{Theorem C}

\theoremstyle{definition}
\newtheorem{definition}[theorem]{Definition}
\newtheorem{remark}[theorem]{Remark}
\newtheorem{example}[theorem]{Example}

\theoremstyle{remark}

\numberwithin{equation}{theorem}

\makeatletter
\renewcommand*\env@matrix[1][\
arraystretch]{%
  \edef\arraystretch{#1}%
  \hskip -\arraycolsep
  \let\@ifnextchar\new@ifnextchar
  \array{*\c@MaxMatrixCols c}}
\makeatother

\renewcommand{\mod}{\operatorname{mod}}
\newcommand{\proj}{\operatorname{proj}}
\newcommand{\inj}{\operatorname{inj}}

\newcommand{\Ext}{\operatorname{Ext}}
\newcommand{\End}{\operatorname{End}}
\newcommand{\Hom}{\operatorname{Hom}}
\newcommand{\add}{\operatorname{\mathrm{add}}}

\newcommand{\domdim}{\operatorname{\mathrm{domdim}}}
\newcommand{\Dom}{\operatorname{\mathrm{Dom}}}
\newcommand{\Codom}{\operatorname{\mathrm{Codom}}}
\newcommand{\Gdim}{\operatorname{\mathrm{Gdim}}}
\newcommand{\fdomdim}{\operatorname{\mathrm{fdomdim}}}
\newcommand{\idim}{\operatorname{\mathrm{idim}}}
\newcommand{\Tr}{\operatorname{\mathrm{Tr}}}

\newcommand{\Gp}{\operatorname{\mathrm{Gp}}}
\newcommand{\Gi}{\operatorname{\mathrm{Gi}}}
\newcommand{\Gpd}{\operatorname{\mathrm{Gpd}}}

\newcommand{\Gpi}{\operatorname{\mathrm{Gpi}}}
\newcommand{\codomdim}{\operatorname{\mathrm{codomdim}}}

\begin{document}

\title{Gendo-symmetric algebras, dominant dimensions and Gorenstein homological algebra}
\date{\today}

\dedicatory{Dedicated to the memory of Leni}



\author{Ren\'{e} Marczinzik}
\address{Institute of algebra and number theory, University of Stuttgart, Pfaffenwaldring 57, 70569 Stuttgart, Germany}
\email{marczire@mathematik.uni-stuttgart.de}

\begin{abstract}
We prove that a finite dimensional algebra $A$ with representation-finite subcategory consisting of modules that are semi-Gorenstein-projective and $n$-th syzygy modules is left weakly Gorenstein.
This generalises a theorem of Ringel and Zhang who proved the result in the case $n=1$.
As an application we show that monomial algebras and endomorphism rings of modules over representation-finite algebras are weakly Gorenstein.
We then give a new connection between the theory of dominant dimension and Gorenstein homological algebra for gendo-symmetric algebras. As an application, we will see that the existence of a non-projective Gorenstein-projective-injective module in a gendo-symmetric algebra already implies that this algebra is not CM-finite.
We apply out methods to give a first systematic construction of non-weakly Gorenstein algebras using the theory of gendo-symmetric algebras. In particular, we can construct non-weakly Gorenstein algebras with an arbitrary number of simple modules from certain quantum exterior algebras such as the Liu-Schulz algebra.
\end{abstract}

\maketitle
\section{Introduction}
We assume always that our algebras are finite dimensional, connected and non-semisimple over a field $K$.
Gorenstein-projective modules are modules $M$ over an algebra $A$ with $\Ext_A^i(M,A)=0$ and $\Ext_A^i(\Tr(M),A)=0$ for all $i>0$ (there $\Tr(M)$ denotes the Auslander-Bridger transpose of a module $M$) and such modules were already studied by Bridger and Auslander in the 1960's in the textbook \cite{AB}. Gorenstein-injective modules are defined dually and we call a module Gorenstein-projective-injective when it is Gorenstein-projective and Gorenstein-injective.
It remained an open problem whether every module $M$ with $\Ext_A^i(M,A)=0$ for all $i>0$ is already Gorenstein-projective for nearly 40 years when a negative answer was finally found in \cite{JS}.
Following \cite{RZ}, we call modules $M$ with  $\Ext_A^i(M,A)=0$ for all $i>0$ semi-Gorenstein-projective.
We introduce left weakly Gorenstein algebras as algebras $A$ such that any semi-Gorenstein-projective module is already Gorenstein-projective. Dually, one can define right weakly Gorenstein algebras and weakly Gorenstein algebras are then defined as algebras that are left and right weakly Gorenstein. We remark that weakly Gorenstein algebras were independently introduced and studied by Ringel and Zhang in \cite{RZ} and in an earlier version of this article we used the name nearly Gorenstein instead of weakly Gorenstein.
In this article we present new connections to other classes of algebras, results, constructions and generalisations of recent results on weakly Gorenstein algebras.
The subcategory $\Omega^n(\mod-A)$ consists of all direct sums of direct summands of $n$-th syzygy modules including all projective modules and the subcategory $^{\perp} A$ is defined as the subcategory of semi-Gorenstein-projective modules. Set $\phi_n:= ^{\perp}A \cap \Omega^n(\mod-A)$ and define an algebra $A$ to be $\phi_n$-finite when this subcategory contains only finitely many indecomposable objects up to isomorphism.
Our first main result in section 3 generalises theorem 1.3 of \cite{RZ}, which is the special case $n=1$.
\begin{theorem*}
Let $A$ be a $\phi_n$-finite algebra for some $n \geq 1$. Then $A$ is left weakly Gorenstein.
\end{theorem*}
As an important corollary of the previous theorem we see will in Corollary \ref{corollarymonrefin} that any monomial algebras is weakly Gorenstein and also any endomorphism ring of a module over a representation-finite algebra is weakly Gorenstein.

In section 4 we will give a new connection between gendo-symmetric algebras and Gorenstein homological algebra.
Gendo-symmetric algebras were introduced by Fang and Koenig in \cite{FanKoe} as endomorphism algebras of generators over symmetric Frobenius algebras. Gendo-symmetric algebras generalise the classical symmetric Frobenius algebras and contain several important classes of finite dimensional algebras such as Schur algebras $S(n,r)$ for $n \geq r$ (see \cite{FanKoe3}), blocks of category $\mathcal{O}$ (see \cite{KSX}) and centraliser algebras of nilpotent matrices (see \cite{CM}).
One of the most important tools to study gendo-symmetric algebras is the (co)dominant dimension of a module $M$, which is defined as the smallest $n$ such that $I_n$ ($P_n$) is not projective (injective) when $I_n$ $(P_n)$ denote the terms in a minimal injective (projective) coresolution of $M$. The dominant dimension of an algebra $A$ is then defined as the dominant dimension of the regular module $A$ and one of the most important open homological conjectures, namely the Nakayama conjecture, states that any non-selfinjective algebra has finite dominant dimension. For the study and relevance of the dominant dimension for gendo-symmetric algebras, we refer for example to \cite{FanKoe} and \cite{FanKoe3}.
A classical result for symmetric Frobenius algebras is that the Auslander-Reiten translate $\tau$ satisfies $\tau(M) \cong \Omega^2(M)$. We will generalise this result in Proposition \ref{tauomega} by showing that in a gendo-symmetric algebra we have that $\tau(M) \cong \Omega^2(M)$ if and only if $M$ has codominant dimension at least two.
Using this characterisation we give the following new connection between the study of dominant dimensions and Gorenstein homological algebra:
\begin{theorem*}
Let $A$ be a gendo-symmetric algebra and $M$ and $A$-module. Then $M$ is Gorenstein-projective-injective if and only if $M$ has infinite dominant and infinite codominant dimension.
\end{theorem*} 
As a consequence of the previous theorem, we will see in Corollary \ref{CorCM} that a gendo-symmetric algebra $A$ having a non-projective-injective Gorenstein-projective-injective module already must have a whole Auslander-Reiten component consisting only of Gorenstein-projective-injective modules and thus the existence of a single such modules implies that the algebra is not CM-finite, that is it has infinitely many indecomposable Gorenstein-projective modules. We will see that one can use our methods to construct such Auslander-Reiten components for gendo-symmetric algebras consisting of Gorenstein-projective-injective modules with the help of tame symmetric Frobenius algebras. An explicit example of such a construction is given with the help of the group algebra of the Klein four group.

In Proposition \ref{nakaconj} we will see that left weakly Gorenstein algebras satisfy several of the classical homological conjectures such as the strong Nakayama conjecture and the Gorenstein symmetry conjecture.
From this viewpoint, it is interesting to find a systematic construction of non-weakly Gorenstein algebras in order to challenge those still open homological conjectures. In the literature it seems that no systematic construction of non-weakly Gorenstein algebras exists. The first example of a non-weakly Gorenstein algebra in \cite{JS} is a local 8-dimensional commutative algebra and Ringel and Zhang found several other local algebras by quiver and relations (see for example \cite{RZ} and \cite{RZ2}) that are not weakly-Gorenstein. In particular, it seems that non-local examples are not present in the literature yet.
Our last main result gives a new construction of non-weakly Gorenstein algebras that gives a first systematic way to obtain non-weakly Gorenstein algebras with a large number of simple modules using some modules with exotic properties.
\begin{theorem*}
Let $A$ be a symmetric Frobenius algebra and $X$ a direct sum of indecomposable non-projective modules. Let $M$ be an indecomposable module such that $\Ext_A^{l}(X,M) \neq 0$ for some $l \geq 1$ and $\Ext_A^i(X,M)=0$ for all $i \geq l+1$. 
Then the gendo-symmetric algebra $B:=\End_A(A \oplus X)$ is not weakly Gorenstein.

\end{theorem*}

We will use the Liu-Schulz algebra, a local symmetric Frobenius algebra named and studied extensively by Ringel in \cite{Rin2}, to construct non-weakly Gorenstein algebras with $n$ simple modules for any $n \geq 2$ that also have the exotic property of having modules of arbitrary large but finite codominant dimensions, see Theorem \ref{theoreminfinitefindomdim}.

I thank Steffen Koenig for useful comments and Ragnar-Olaf Buchweitz for telling me about the results in \cite{JS}.

\section{Preliminaries}
\subsection{General preliminaries}
We start by fixing some notations and giving definitions. For background on Auslander-Reiten theory and homological algebra we refer for example to \cite{SkoYam}. 
Let an algebra always be a finite dimensional, connected and non-semisimple algebra over a field $K$ and a module over such an algebra is always a finite dimensional right module, unless otherwise stated. We will also assume that $A$ is a basic algebra. Note that this is no restriction on generality since all notions and homological dimensions we deal with in this article are invariant under Morita equivalence. $D=\Hom_K(-,K)$ denotes the duality for a given finite dimensional algebra $A$ and $J$ will denote the Jacobson radical.
We define the \emph{dominant dimension} $\domdim M$ of a module $M$ with a minimal injective resolution 
$$ 0 \rightarrow M \rightarrow I_0 \rightarrow I_1 \rightarrow ...$$  as the smallest $n$ such that $I_n$ is not projective, and as infinite if no such $n$ exists.
The \emph{codominant dimension} of a module $M$ is defined as the dominant dimension of the dual module $D(M)$. The dominant dimension of a finite dimensional algebra is defined as the dominant dimension of the regular module $A_A$ and the codominant dimension is the codominant dimension of the module $D(_AA)$. \newline
It is well known that $A$ has dominant dimension larger than or equal to 1 if and only if there exists an idempotent $e$ such that $eA$ is a minimal faithful projective-injective right module if and only if there exists an idempotent $f$ such that $Af$ is a minimal faithful projective-injective left module, see section 3.1 in \cite{Yam}. Note that the dominant dimension of $A_A$ is always equal to the dominant dimension of $_AA$ and thus the dominant dimension of an algebra always equals the codominant dimension, see proposition 3.4 in \cite{Yam}. $eA$ always denotes the minimal faithful projective-injective right $A$-module and $Af$ denotes the minimal faithful projective-injective left $A$-module for some idempotents $e$ and $f$. \newline
For $1 \leq i \leq \infty$ $\Dom_i(A)$ (resp. $\Codom_i(A)$) denotes the full subcategory of mod-$A$ constisting of modules with dominant dimension (resp. codominant dimension) larger than or equal to $i$. We often use the notation $\Dom(A):=\Dom_{\infty}(A)$ and $\Codom(A):=\Codom_{\infty}(A)$. Let $\proj(A)$ denote the full subcategory of finitely generated projective modules and $\inj(A)$ the subcategory of finitely generated injective modules. 
An algebra $A$ is called \emph{$g$-Gorenstein} for a natural number $g$ if the left and right injective dimensions of $A$ coincide and are equal to $g$. We call $A$ \emph{Gorenstein}, if $A$ is $g$-Gorenstein for some $g$. In this case $\Gdim A:=\idim A$ is called the Gorenstein dimension of $A$.
$\nu_A=\nu=D \Hom_A(-,A)$ denotes the Nakayama functor and $\nu^{-1}=\Hom_A(D(-),A)$ its inverse. For section 2 and 4 of this article, fix the notations $I:=eA$ for the minimal faithful projective-injective $A$-module, $P:=\nu_A^{-1}(eA)=\Hom_A(D(eA),A)$, $B_1:=\End_A(I)$ and $B_2:=\End_A(P)$. We note that $B_1$ and $B_2$ are isomorphic if the dominant dimension of $A$ is at least two, see for example \cite{Yam}, after theorem 3.4.1. As a right $A$-module $P$ is isomorphic to $fA$. We call a module $W$ \emph{$l$-periodic} for a natural number $l \geq 1$ if $\Omega^{l}(W) \cong W$. For $i \geq 1$, define the set of \emph{$i$-torsionless modules} as $\{ X | \Ext_A^{l}(\Tr(X),A)=0$ for $l=1,...,i \}$. Recall that $2$-torsionless modules are exactly the reflexive modules. \newline
For $n \in \mathbb{N}$, denote by $\Omega^{n}(\mod-A)$ the full subcategory of all direct sums of modules, which are summands of $n-$th syzygy modules including all projective modules. We say that an algebra $A$ is $\Omega^i$-finite for some $i \geq 1$, if the subcategory $\Omega^i(\mod-A)$ has only finitely many indecomposable modules up to isomorphism.
For an $A$-module $M$, $\add(M)$ denotes the full subcategory of $\mod-A$ consisting of all direct summands of a finite direct sum of $M$. Then a map $f:M_0 \rightarrow X$, with $M_0 \in$ $\add(M)$ is called a \emph{right add($M$)-approximation} of $X$ if and only if the induced map $\Hom(N,M_0) \rightarrow \Hom(N,X)$ is surjective for every $N \in$ $\add(M)$.
Note that in case $M$ is a generator, such an $f$ must be surjective.
When $f$ is also a right minimal homomorphism, we call it a \emph{minimal right $\add(M)$-approximation} (or $\add(M)$-Cover).
Note that minimal right $\add(M)$-approximations always exist for finite dimensional algebras.
The kernel of such a minimal right $\add(M)$-approximation $f$ is denoted by $\Omega_M(X)$. Inductively one defines $\Omega_M^{0}(X):=X$ and $\Omega_M^{n}(X):=\Omega_M(\Omega_M^{n-1}(X))$.
The \emph{$\add(M)$-resolution dimension} of a module $X$ is defined as
$M\text{-resdim}(X):=\inf \{ n \geq 0 | \Omega_M^{n}(X) \in \add(M) \}.$
Left approximations and coresolution dimensions are defined dually.
If $C$ is a subcategory of $\mod-A$, we denote by $C/[\proj]$ the subcategory modulo projectives, called stable category, and $C/[\inj]$ the subcategory modulo injectives, called costable category. 
\setcounter{subsection}{1}
The next definition due to Fang and Koenig generalises the classical symmetric algebras $A$, that are defined by the condition $A \cong D(A)$ as $A$-bimodules.
\begin{definition}
$A$ is called a \emph{gendo-symmetric algebra} if and only if it has dominant dimension larger than or equal to 2 and $D(Ae) \cong eA$ as $(eAe,A)-$bimodules. This is equivalent to $A$ being isomorphic to $\End_B(M)$, where $B$ is a symmetric algebra and $M$ a generator of mod-$B$ (see \cite{FanKoe}).
\end{definition}
For other characterisations of gendo-symmetric algebras we refer to \cite{FanKoe} and \cite{Mar2}.
Note that if the algebra is gendo-symmetric, $I=P$ and $B_1=B_2$.
For a subcategory $\mathcal{X}$ of mod-$A$ we define for $1 \leq n \leq \infty$: $\mathcal{X}^{\perp n}:= \{ M \in \mod-A | \Ext_A^{i}(Y,M)=0 $ for $1 \leq i \leq n $ and all $Y \in \mathcal{X} \}$ and $^{\perp n}\mathcal{X}:= \{ M \in \mod-A | \Ext_A^{i}(M,Y)=0 $ for $1 \leq i \leq n $ and all $Y \in \mathcal{X} \}$. We often use the shorter notation $\mathcal{X}^{\perp} := \mathcal{X}^{\perp \infty}$ and $^{\perp }\mathcal{X} :=^{\perp \infty}\mathcal{X}$. If $\mathcal{X}=\add(X)$ for a module $X$ we write  $X^{\perp n}$ instead of  $\mathcal{X}^{\perp n}$ and similar for the other notations involving  $\mathcal{X}$. 
The following theorem collects results from \cite{APT} in a special case.
\begin{theorem}
\phantomsection
\label{ARSmaintheorem}
\begin{enumerate}
\item The functors $F_1:=\Hom_A(I,-) : \Codom_2(A) \rightarrow \mod-B_1$ and $F_2:=\Hom_A(P,-) : \Dom_2(A) \rightarrow \mod-B_2$ are equivalences of categories.
$F_1$ restricts to an equivalence between $\add(I)$ and the category of projective $B_1$-modules and $F_2$ restricts to an equivalence between $\add(I)$ and the category of injective $B_2$-modules.
\item There are natural isomorphisms of functors: \newline
$\Hom_A(P,-) \cong (-) \otimes_A (D(eA)) \cong  D \Hom_A(-,D(eA))$.
If $A$ is gendo-symmetric the following holds true: $F_2 \cong (-)e $ and $(-)e \cong F_1$.
\item The functor $G_1:= (-) \otimes_{B_1} I : \mod-B_1 \rightarrow \Codom_2(A)$ is inverse to $F_1$ and the functor $G_2:= \Hom_{B_2}(P,-) : \mod-B_2 \rightarrow \Dom_2(A)$ is inverse to $F_2$.
\item For $i \geq 3$, $F_1$ restricts to an equivalence $F_1: \Codom_i(A) \rightarrow ^{\perp i-2}(Ae)$ and $F_2$ restricts to an equivalence $F_2: \Dom_i(A) \rightarrow (Af)^{\perp i-2}$.
\end{enumerate}
\end{theorem}
\begin{proof}
\begin{enumerate}
\item This is a special case of Lemma 3.1 of \cite{APT}.
\item By \cite{SkoYam} Chapter III Lemma 6.1, there is the following natural isomorphism of functors:  \newline $\Hom_A( \Hom_A(D(eA),A),-) \cong (-) \otimes_A (D(eA))$. Now there is another natural isomorphism $(-) \otimes_A (D(eA)) \cong D \Hom_A(-,D(eA))$ by \cite{ASS} Appendix 5, Proposition 4.11.
When $A$ is gendo-symmetric $F_2(-) \cong (-) \otimes_A (D(eA)) \cong (-) \otimes_A Ae \cong (-)e$, is clear since $D(eA) \cong Ae$ as $(A,eAe)$-bimodules. 
\item This follows from \cite{APT}, in the passages before Proposition 3.9 and before Proposition 3.10.
\item This follows from Proposition 3.7 in \cite{APT}. 
\end{enumerate}
\end{proof}
One can use (4) of the previous theorem to calculate dominant dimensions via Ext.

We recall the following well known theorem of Mueller, see \cite{Mue}:
\begin{theorem} \label{mueller}
Let $A$ be an algebra with generator-cogenerator $M$ and $B:=\End_A(M)$. Let $N$ be an $A$-module.
Then the $B$-module $\Hom_A(M,N)$ has dominant dimension equal to $\inf \{ i \geq 1 | \Ext_A^{i}(M,N) \neq 0 \}+1$.
Especially: The dominant dimension of $B$ equals $\inf \{ i \geq 1 | \Ext_A^{i}(M,M) \neq 0 \}+1$.
\end{theorem}

\begin{definition}
The finitistic dominant dimension of an algebra $A$ is defined as
$\fdomdim(A):= \sup \{ \domdim M | M \in \mod-A$ and $\domdim(M) < \infty \}$.
\end{definition}
The finitistic codominant dimension is defined dually.
In \cite{Mar}, we proved that the finitistic dominant dimension of a Nakayama algebra with $n$ simple modules is bounded by $2n-2$. 
\begin{lemma}
\label{torsionless}
Assume $A$ has dominant dimension $d \geq 1$.
\begin{enumerate} 
\item $\Dom_i(A) = \Omega^{i}(\mod-A)$ for all $1 \leq i \leq d$ and $\Codom_i(A) = \Omega^{-i}(\mod-A)$ for all $1 \leq i \leq d$.
\item $\Dom_i(A)=\{ X | \Ext_A^{l}(\Tr(X),A)=0$ for $l=1,...,i \}$, for $i=1,...,d$.
\end{enumerate}
\end{lemma}
\begin{proof}
\begin{enumerate}
\item See \cite{MarVil} proposition 4.
\item Combine proposition 1.6 of \cite{AR} and part (1) of this lemma.
\end{enumerate}
\end{proof}

The following Lemma is proven in Direction 1 of \cite{Rin2}.
\begin{lemma}
\label{direction}
Let $M$ be a generator and cogenerator of $B$. Let $A:= \End_B(M)$. Then the basic versions of $Af$ and $M$ are isomorphic as $B \cong fAf$-modules, where $Af$ is the minimal faithful projective-injective left $A$-module.
\end{lemma}

The following result is proposition 3.11 from \cite{CheKoe}:
\begin{theorem}
\label{theoremchekoe}
Let $A$ be a finite dimensional algebra and $M$ a non-projective generator and cogenerator of mod-$A$ and define $B:=\End_A(M)$.
Let $B$ have dominant dimension $z+2$, with $z \geq 0$.
Then, for the right injective dimension of $B$ the following holds:
$$\idim(B_B)=z+2\ +\ M\text{-resdim}(\tau_{z+1}(M)\oplus D(A)).$$
For the left injective dimension of $B$ the following holds:
$$\idim(_BB)=z+2\ +\ M\text{-coresdim}(\tau_{-(z+1)}(M)\oplus A).$$
\noindent Here we use the notations $\tau_{z+1}=\tau\Omega^{z}$ and $\tau_{-(z+1)}=\tau^{-1}\Omega^{-z}$, introduced by Iyama (see \cite{Iya}).
\end{theorem}

The next proposition collects several homological results for symmetric algebras, that we will need later.
\begin{proposition}
\label{extrechnen}
Let $B$ be a symmetric algebra and $M,N$ two $B$-modules.
\begin{enumerate}
\item For the Auslander-Reiten translate $\tau$, the following holds: $\tau(M) \cong \Omega^{2}(M)$, for every indecomposable non-projective module $M$.
\item $\Ext_A^{i}(M,N) \cong \underline{\Hom}_A(\Omega^{i}(M),N) \cong \underline{\Hom}_A(M,\Omega^{-i}(N))$.
\item $\Ext_A^{1}(M,N) \cong \underline{\Hom}_A(N,\Omega^{2}(M)) \cong \underline{\Hom}_A(\Omega^{-2}(N),(M))$
\item $\underline{\Hom}_A(M,N) \cong \underline{\Hom}_A(N,\Omega^{1}(M))$.
\end{enumerate}
\end{proposition}
\begin{proof}
\begin{enumerate}
\item See \cite{SkoYam}, Chapter IV Corollary 8.6.
\item See \cite{SkoYam}, Chapter IV Theorem 9.9.
\item Those are the Auslander-Reiten formulas in the special case of a symmetric algebra, see \cite{SkoYam} Chapter III. Theorem 6.3.
\item This follows from the previous part using that $\Omega$ is an equivalence of the stable module category of $B$ with inverse $\Omega^{-1}$: \newline
$\underline{\Hom}_A(N,\Omega^{1}(M)) \cong \underline{\Hom}_A(\Omega^{1}(N),\Omega^{2}(M)) \cong \Ext_A^{1}(M,\Omega^{1}(N)) \newline \cong \underline{\Hom}_A(\Omega^{1}(M),\Omega^{1}(N)) \cong \underline{\Hom}_A(M,N).$
\end{enumerate}
\end{proof}
We also need the following well known lemma, which can be found in \cite{Ben} as Corollary 2.5.4.:

\begin{lemma}
\label{benson}
Let $A$ be a finite dimensional algebra and $N$ be an indecomposable $A$-module and $S$ a simple $A$-module.
Let $(P_i)$ be the terms of a minimal projective resolution of $N$ and $(I_i)$ the terms of a minimal injective resolution of $N$.
\begin{enumerate}
\item For $l \geq 0$, $\Ext_A^{l}(N,S) \neq 0$ if and only if $S$ is a quotient of $P_l$.
\item For $l \geq 0$, $\Ext_A^{l}(S,N) \neq 0$ if and only if $S$ is a submodule of $I_l$.
\end{enumerate}
\end{lemma}

\subsection{Gorenstein-projective and Gorenstein-injective modules}
In this subsection we recall the most important notions and results from Gorenstein homological algebra. For more background on Gorenstein homological algebra, we refer to \cite{Che} and \cite{AB}.
\begin{definition}
A module $M$ is called \emph{Gorenstein-projective}, if $\Ext_A^{i}(M,A)=0=\Ext_A^{i}(\Tr(M),A)$ for every $i \geq 1$. $M$ is called \emph{Gorenstein-injective} if $D(M)$ is Gorenstein-projective. If $M$ is Gorenstein-projective and Gorenstein-injective, we call $M$ \emph{Gorenstein-projective-injective}. We refer to \cite{Che} and \cite{AB} for many other characterisations and properties of Gorenstein-projective and Gorenstein-injective modules.
For a finite dimensional algebra $A$, we denote by $\Gp(A)$ the category of finite dimensional Gorenstein-projective modules. We denote by $\Gi(A)$ the category of finite dimensional Gorenstein-injective modules. \newline $\Gpi(A)$ is defined as the subcategory of finite dimensional Gorenstein-projective-injective modules.
$A$ is called \emph{CM-finite} if $\Gp(A)$ is representation-finite and $A$ is called \emph{CM-free} if $\Gp(A)=\proj(A)$. Recall that a Gorenstein algebra is CM-free if and only if it has finite global dimension.
\end{definition}
We remark that every Gorenstein-projective module is an $n$-th syzygy module for any $n \geq 1$.

The following two propositions can be found in \cite{Che} as proposition 2.2.3 and theorem 2.3.3:
\begin{proposition}
\label{taugorenstein}
\begin{enumerate}
\item A module $M$ is Gorenstein-injective if and only if $\nu^{-1}(M)$ is Gorenstein-projective and the natural morphism $\nu\nu^{-1}(M) \rightarrow M$ is an isomorphism.
\item There is an equivalence of categories: $\nu: \Gp(A) \rightarrow \Gi(A)$ with quasi-inverse $\nu^{-1}$.
\end{enumerate}
\end{proposition}

\begin{proposition}
\label{gorensteinkrit}
The following are equivalent:
\begin{enumerate}
\item $A$ is Gorenstein of Gorenstein dimension $g$.
\item $\Gp(A)= \Omega^{g}(\mod-A)$.
\item $\Gi(A)= \Omega^{-g}(\mod-A)$.
\end{enumerate}
When $A$ is Gorenstein, the Gorenstein-projective modules coincide with the modules $M$ with $\Ext_A^{i}(M,A)=0$ for all $i \geq 1$ and the Gorenstein-injective modules coincide with the modules $M$ with $\Ext_A^{i}(D(A),M)=0$ for all $i \geq 1$.
\end{proposition}

\begin{proposition}
\label{gorensteingrund}
A module $M$ is Gorenstein-projective if and only if $\tau(M)$ is Gorenstein-injective and a module $N$ is Gorenstein-injective if and only if $\tau^{-1}(N)$ is Gorenstein-projective.

\end{proposition}
\begin{proof}
This is Proposition 2.2.13 in \cite{Che}.

\end{proof}
\section{Weakly Gorenstein algebras}

\subsection{Basics and examples for weakly Gorenstein algebras}
\begin{definition}
We call an algebra $A$ \emph{right weakly Gorenstein}, if the subcategory of Gorenstein-injective modules coincides with $D(A)^{\perp}$, the category of \emph{semi-Gorenstein-injective modules}. We call an algebra $A$ \emph{left weakly Gorenstein}, if the subcategory of Gorenstein-projective modules coincides with $^{\perp}A$, the category of \emph{semi-Gorenstein-projective modules}. We call an algebra $A$ \emph{weakly Gorenstein}, if $A$ is left and right weakly Gorenstein.
\end{definition}

\begin{example}\label{gorexamp}
Every Gorenstein algebra is weakly Gorenstein by \ref{gorensteinkrit}, hence the name.
By \cite{Mar} Lemma 1.2.3 and its dual, every Nakayama algebra is weakly Gorenstein. But not every Nakayama algebra is Gorenstein, take for example the Nakayama algebra with Kupisch series $[3,4]$.
\end{example}

Here is a natural characterisation of weakly Gorenstein algebras that does not mention Gorenstein-projective modules:
\begin{proposition}
An algebra $A$ is weakly Gorenstein if and only if there is an equivalence of categories induced by the Auslander-Reiten translate \newline
$\tau: ^{\perp}\!A/[\proj] \rightarrow D(A)^{\perp}/[\inj]$ with inverse $\tau^{-1}: D(A)^{\perp}/[\inj] \rightarrow ^{\perp}\!A/[\proj]$.
\end{proposition}
\begin{proof}
Recall that $\tau: \mod-A/[\proj] \rightarrow \mod-A/[\inj]$ is an equivalence with inverse $\tau^{-1}$.
Assume first that $A$ is weakly Gorenstein. Then $^{\perp}\!A=\Gp(A)$ and $D(A)^{\perp}=\Gi(A)$ and the result follows since $\tau(X)$ is Gorenstein-injective, if $X$ is Gorenstein-projective and vice versa by \ref{gorensteingrund}. Thus $\tau$ restricts to an equivalence:  $^{\perp}\!A/[\proj] \rightarrow D(A)^{\perp}/[\inj]$. \newline
Now assume that $\tau: ^{\perp}\!A/[\proj] \rightarrow D(A)^{\perp}/[\inj]$ is an equivalence with inverse \newline $\tau^{-1}: D(A)^{\perp}/[\inj] \rightarrow ^{\perp}\!A/[\proj]$. We will show $^{\perp}\!A=\Gp(A)$.
Assume that $X$ is a non-projective indecomposable $A$-module and $X \in ^{\perp}\!A$. Then $\tau(X) \in D(A)^{\perp}$. Thus $\Ext_A^{i}(D(A),\tau(X))=0$ for all $i \geq 1$, which is equivalent to $\Ext_A^{i}(Tr(X),A)=0$ for all $i \geq 1$.
Thus, by definition, $X$ is Gorenstein-projective. $D(A)^{\perp}=\Gi(A)$ follows dually, and thus $A$ is weakly Gorenstein. 
\end{proof}

Recall that the \emph{Gorenstein-projective dimension} $\Gpd(M)$ of a module $M$ is defined as the smallest $n$ such that there exists an exact sequence with (possibly infinite dimensional) Gorenstein-projective modules $G^i$
$$0 \rightarrow G^n \rightarrow \cdots G^1 \rightarrow G^0 \rightarrow M \rightarrow 0.$$
The Gorenstein-projective dimension is defined to be infinite if no such finite exact sequence exists for $M$.
We refer for example to \cite{Che} for more on the Gorenstein-projective dimension.
Recall that a module $M$ is called periodic if $\Omega^i(M) \cong M$ for some $i>0$.
\begin{proposition} \label{propohelp}
Let $A$ be a finite dimensional algebra with modules $M$ and $N$.
\begin{enumerate}
\item If $M$ has finite Gorenstein-projective dimension, we have $\Gpd(M)= \sup \{ t \geq 0 | \Ext_A^t(M,A) \neq 0 \}$.
\item If $M$ is non-projective indecomposable semi-Gorenstein-projective, $\Omega^k(M)$ is indecomposable for all $k>0$.
\item If $M$ is semi-Gorenstein-projective, we have $\Ext_A^i(M,N) \cong \underline{\Hom}_A(\Omega^i(M),N)$ for all $i>0$.
\item A semi-Gorenstein-projective and periodic module is Gorenstein-projective.
\item If $M$ is semi-Gorenstein-projective, we have $\underline{\Hom}_A(M,N) \cong \underline{\Hom}_A(\Omega^i(M),\Omega^i(N))$ for all $i>0$.
\end{enumerate}

\end{proposition}
\begin{proof}
\begin{enumerate}
\item See \cite{Che}, proposition 3.2.2.
\item See \cite{RZ}, corollary 3.3.
\item See \cite{Che}, lemma 2.1.8.
\item See \cite{Che}, proposition 2.2.17.
\item See \cite{Iya}, in section 2.1.
\end{enumerate}
\end{proof}

We define the full subcategory $\phi_n:= ^{\perp} A \cap \Omega^n(\mod-A)$ of $\mod-A$ and call an algebra $A$ \emph{$\phi_n$-finite} if this subcategory has only finitely many indecomposable modules up to isomorphism.

\begin{theorem} \label{main result}
Let $A$ be a $\phi_n$-finite algebra for some $n \geq 1$. Then $A$ is left weakly Gorenstein.

\end{theorem}
\begin{proof}
Note that $\phi_k(A) \subseteq \phi_l(A)$ for $k \geq l$ and thus $A$ is also $\phi_r-$finite for all $r \geq n$.
Assume $M$ is semi-Gorenstein-projective. We will show that $M$ is even Gorenstein-projective. This is trivial if $M$ is projective and thus we assume $M$ is non-projective and indecomposable in the following.
Now $\Ext_A^i(\Omega^n(M),A)=\Ext_A^{i+n}(M,A)=0$ for all $i>0$ shows that $M$ being semi-Gorenstein-projective also gives that $\Omega^n(M)$ is semi-Gorenstein-projective. Thus $\Omega^n(M) \in \phi_n(A)$. Now by \ref{propohelp} (2), $\Omega^k(M)$ are all indecomposable modules for $k \geq 1$. 
Since $\phi_n(A)$ contains only finitely many indecomposable modules and $\Omega^k(M)$ is indecomposable for all $k \geq n$, we must have that $\Omega^{l+r}(M) \cong \Omega^l(M)$ for some $l \geq n$ and $r \geq 1$.
This shows that $\Omega^l(M)$ is a periodic module. By \ref{propohelp} (4), $\Omega^l(M)$ is Gorenstein-projective.
Let $(P_i)$ for $i \geq 0$ be a minimal projective resolution of $M$.
Then the exact sequence
$$0 \rightarrow \Omega^l(M) \rightarrow P_{l-1} \rightarrow \cdots \rightarrow P_0 \rightarrow M \rightarrow 0$$
shows that $M$ has finite Gorenstein-projective dimension at most $l$. 
Now assume $M$ has non-zero Gorenstein-projective dimension.
By \ref{propohelp} (2), $\Gpd(M) =  \sup \{ t \geq 0 | \Ext_A^t(M,A) \neq 0 \}$ (since we know that the Gorenstein-projective dimension of $M$ is finite), which gives a contradiction to our assumption that $M$ is semi-Gorenstein-projective. Thus $M$ must have Gorenstein-projective dimension zero and thus is Gorenstein-projective.
\end{proof}

\begin{corollary} \label{corollarymonrefin}
Let $A$ be a finite dimensional algebra.
\begin{enumerate}
\item If $A$ is a monomial quiver algebra, $A$ is weakly Gorenstein.
\item If $A$ is isomorphic to the endomorphism ring of a module over a representation-finite algebra, $A$ is weakly Gorenstein.

\end{enumerate}

\end{corollary}
\begin{proof}
\begin{enumerate}
\item By theorem 1 of \cite{Z}, monomial algebras are $\Omega^2$-finite and thus also $\phi_2$-finite. By \ref{main result} this shows that monomial algebras are left weakly Gorenstein. Since the opposite algebra of a monomial algebra is again monomial, it is also right weakly Gorenstein and thus weakly Gorenstein.
\item We first show that endomorphism rings of modules over representation-finite algebras are $\Omega^2$-finite and thus also $\phi_2$-finite. Let $A=\End_B(M)$ for a representation-finite algebra $B$ and a $B$-module $M$. 
Note that the second syzygy modules are exactly the kernels of maps $P_1 \rightarrow P_0$ between projective modules. Recall that there is an equivalence $\add(M) \rightarrow \proj-A$ given by $\Hom_B(M,-)$, see for example \cite{ARS}, proposition 2.1 in chapter II.
Under this equivalence, a map $f:P_1 \rightarrow P_0$ between projective $A$-modules 
corresponds to a map $g: M_1 \rightarrow M_0$ with $M_i \in \add(M)$. Since $B$ is representation-finite the possible kernels of such maps $g$ can have only finitely many indecomposable summands and thus there are also only finitely many possible indecomposable summands of maps $f:P_1 \rightarrow P_0$ between projective $A$-modules.
This shows that $A$ is $\Omega^2$-finite and thus also $\phi_2$-finite.
By our main result \ref{main result} this shows that endomorphism algebras of modules over representation-finite algebras are left weakly Gorenstein. Since the opposite algebra of $\End_A(M)$ is isomorphic to $\End_{A^{op}}(D(M))$ and $A^{op}$ is again representation-finite, such algebras are also right weakly Gorenstein and thus weakly Gorenstein.

\end{enumerate}
\end{proof}

The next example shows that the algebras $\End_B(M)$ can be of wild representation type in general even when $B$ is representation-finite and thus it might be quite complicated to prove that such endomorphism algebras are weakly Gorenstein directly.

\begin{example}
This example was found and calculated using the GAP-package QPA, see \cite{QPA}.
Let $K$ be the field with three elements.
Let $A:=KQ/I$ with $Q$ and $I$ as follows:
$$
\begin{xy}
  \xymatrix{
      & {\bullet}^1\ar@(ul,dl)_{\alpha} \ar@/ ^1pc/[r]^{\beta_1}   & {\bullet}^2 \ar@/ ^1pc/[l]^{\beta_2}}
\end{xy} \newline 
\\\ , \ I=<\alpha^2-\beta_1 \beta_2, \beta_2 \beta_1>.
$$ 
$A$ is a representation-finite symmetric algebra, see for example \cite{Sko} section 3.14.
Let $M$ be the following representation of $A$ (given by its representation in QPA):
$$
\begin{xy}
  \xymatrix{
      & {\bullet}^{k^4}\ar@(ul,dl)_{M_{\alpha}} \ar@/ ^1pc/[r]^{M_{\beta_1}}   & {\bullet}^{k^2} \ar@/ ^1pc/[l]^{M_{\beta_2}}}
\end{xy} $$ 
 with $M_{\alpha}= \begin{pmatrix}[1]
 0 & 0 & 0 & 0\\
  1 & 0 & 0 &0\\
  0 & 1 & 0 & 0 \\
  1 &0&0& 0\\
 \end{pmatrix}$, $M_{\beta_1}= \begin{pmatrix}[1]
 0 & 0 \\
  0 & -1 \\
  1 & 0  \\
  0 &0\\
 \end{pmatrix}$ and $M_{\beta_2}= \begin{pmatrix}[1]
 1 & 0 &0&0 \\
  0 & 0&0&0 \\
 \end{pmatrix}.$ \newline
 One can check that $M$ is indecomposable and $C:=\End_A(M)$ is isomorphic to the quiver algebra $KW/L$ where $W$ is the quiver with one point and three loops $a_1, a_2$ and $a_3$ and $L=<a_1^2,a_1 a_3,a_2 a_1, a_2^2,a_2 a_3, a_1 a_2-a_3a_1,a_1 a_2-a_3a_2, a_3^2>$.
$C$ is a local non-Gorenstein algebra. Since it is the endomorphism ring of a module over a representation-finite algebra, it is weakly Gorenstein by \ref{corollarymonrefin}.
\end{example}

Since our main concern will be Gorenstein homological algebra of gendo-symmetric algebras in the next section, we give here a concrete example of a gendo-symmetric weakly Gorenstein algebra that is not Gorenstein. Here we use \hyperref[theoremchekoe]{ \ref*{theoremchekoe}}.
\begin{example}
We choose $A$ to be the symmetric Nakayama algebra with Loewy length $7$ and $3$ simple modules and we set $M=e_0 J^2$. The algebra $B:=\End_A(W)$, with $W:= A \oplus M$, is weakly Gorenstein by (2) of the previous theorem. It is easy to see that $\Ext_A^{1}(M,M)=0$, but $\Ext_A^{2}(M,M)\neq 0$ and so $\domdim B=3$. $\tau_2(M)=\tau(\Omega^{1}(M))= \tau(e_2 J^{5})=e_0 J^{5}$ and thus we have to calculate the right $W$-resolution dimension of $ e_0 J^{5}$ by \ref{theoremchekoe}.
First one calculates the start of this $W$-resolution and the first kernel. Note first that because of $\Ext_A^{1}(e_0 J^2,e_0 J^2)=0$, the subcategory $\mathbb{X}:=\add(A \oplus e_0 J^2)$ is extension-closed. Then by Wakamatsus lemma (see \cite{EJ} Chapter 7.2), a map $f:A \rightarrow B$ is an $\mathbb{X}-$approximation if and only if $A \in \mathbb{X}$ and $\Ext_A^{1}(L,ker(f))=0$ for every $L \in \mathbb{X}$. The minimal $\add(W)$-cover $\pi$ of $e_0 J^5$ looks as follows:
$$0 \rightarrow e_0 J^4 \rightarrow e_0 J^{2} \stackrel{\pi}\rightarrow e_0 J^5 \rightarrow 0.$$
\noindent $e_0J^2 \cong e_0 A / e_0 J^5$ and $e_0 J^5 \cong e_2 A/ e_2 J^2$ and $\pi$ is the surjective map which is left multiplication by the unique arrow of length $1$ from $e_2$ to $e_0$.  
Now we calculate a minimal $W$-cover $\pi_2$ of the kernel $e_0 J^4$. We have
$$0 \rightarrow e_0J^4 \oplus e_1 J \rightarrow e_1 A \oplus e_0 J^2 \stackrel{\pi_2}\rightarrow e_0 J^4 \rightarrow 0.$$
Note that $e_0 J^4 \cong e_1 A/e_1 J^3$ and $\pi_2=(f,g)$, where, $f: e_1 A \rightarrow e_0J^4$ is the projective cover of $e_0 J^4$ and $g: e_0 A /e_0J^4 \rightarrow e_0 J^4$ is left multiplication by the unique arrow of length $2$ from $e_1$ to $e_0$.
Now it is clear that the minimal $W$-resolution dimension of $e_0 J^5$ has to be infinite, since the second kernel of the resolution has the first kernel as a direct summand. By \hyperref[theoremchekoe]{ \ref*{theoremchekoe}}, $B$ is not Gorenstein.
\end{example}
In \cite{Mar}, we proved that all gendo-symmetric Nakayama algebras are Gorenstein. Therefore, there are probably not much easier examples of non-Gorenstein, gendo-symmetric algebras than the previous one.
\subsection{Homological conjectures for weakly Gorenstein algebras}
Recall the following famous homological conjectures for a finite dimensional algebra $A$ (see for example \cite{Yam} for a discussion of some of the conjectures):
\begin{enumerate}
\item Strong Nakayama conjecture: \newline
For every non-zero module $M$ there is an $i \geq 0$ with $\Ext_A^{i}(M,A) \neq 0$.
\item The generalized Nakayama conjecture: \newline
For every simple module $S$ there is an $i \geq 0$ with $\Ext_A^{i}(S,A) \neq 0$.
\item The Nakayama conjecture:
Every non-selfinjective algebra has finite dominant dimension.
\item The Gorenstein symmetry conjecture: \newline
The right injective dimension equals the left injective dimension of an algebra. \newline \newline 
It is well-known and easy to see that $(1) \Rightarrow (2) \Rightarrow (3)$.
\end{enumerate}
\begin{proposition} \label{nakaconj}
\begin{enumerate}
\item The strong Nakayama is true for every left weakly Gorenstein algebra and thus every non-selfinjective left weakly Gorenstein algebra has finite dominant dimension.
\item The Gorenstein symmetry conjecture is true for left weakly Gorenstein algebras.
\end{enumerate}
\end{proposition}
\begin{proof}
\begin{enumerate}
\item
Let $M$ be a module with $\Ext_A^{i}(M,A)=0$ for every $i \geq 1$. Then $M$ is Gorenstein-projective, since $A$ is left weakly Gorenstein. But then there is an embedding $M \rightarrow A^n$ (since $M$ is a torsionless modules as a Gorenstein-projective module), for some $n \geq 1$. Thus $\Hom_A(M,A) \neq 0$.
\item If the right injective dimension of $A_A$ is zero, then $A$ is selfinjective and being selfinjective is left-right symmetric.
So assume that the right injective dimension of $A_A$ is $n \geq 1$. Then for every module $X: \newline \Ext_A^{n+i}(X,A)=0$ for every $i \geq 1$. Now for every $i \geq 1: \Ext_A^{n+i}(X,A) \cong \Ext_A^{i}(\Omega^{n}(X),A)=0$ and thus $\Omega^{n}(X) \in ^{\perp}\!A=\Gp(A)$, using that $A$ is left weakly Gorenstein. This gives us $\Omega^{n}(\mod-A) \subseteq \Gp(A)$. But by definition of Gorenstein-projective $\Gp(A) \subseteq \Omega^{j}(\mod-A)$ for every $j \geq 1$ and thus $\Omega^{n}(\mod-A) = \Gp(A)$, which by \hyperref[gorensteinkrit]{ \ref*{gorensteinkrit}}  means that $A$ is Gorenstein with Gorenstein dimension $n$ and thus also has left injective dimension $n$.
\end{enumerate}
\end{proof}

Recall that an algebra $A$ is called \emph{CM-finite} if $\Gp(A)$ is representation-finite. 
Note that since $\Gp(A) \subseteq \phi_n(A)$ for all $n \geq 1$, being $\phi_n$-finite implies being CM-finite. 
The Auslander-Reiten conjecture states that a module $M$ with $\Ext_A^i(M,M \oplus A) =0$ for all $i>0$ is projective (see for example conjecture 10 in the conjectures section in \cite{ARS}).
While it is known that the finitistic dimension conjecture implies the Auslander-Reiten conjecture, it is not known whether the Auslander-Reiten conjecture holds for a fixed algebra with finite finitistic dimension. For example selfinjective algebras have finitistic dimension equal to zero, but the Auslander-Reiten conjecture is open for selfinjective algebras.
\begin{proposition}
Let $A$ be a left weakly Gorenstein algebra that is CM-finite. Then the Auslander-Reiten conjecture holds for $A$.

\end{proposition}
\begin{proof}
Assume there exists a non-projective indecomposable module $M$ with $\Ext_A^i(M,M \oplus A) =0$ for $i>0$. 
This gives that $\Ext_A^i(M,A) =0$ for $i>0$ and thus $M$ is Gorenstein-projective since $A$ is by assumption left weakly Gorenstein.
Now since $A$ is CM-finite, there are only finitely many Gorenstein-projective $A$-modules and since with $M$ also $\Omega^k(M)$ are Gorenstein-projective for all $k>0$, there must exist integers $l,t>0$ with $\Omega^{l+t}(M) \cong \Omega^l(M)$.
This gives that $\Ext_A^t(\Omega^l(M),\Omega^l(M)) \cong \underline{\Hom}_A(\Omega^{l+t}(M),\Omega^l(M)) \cong \underline{\Hom}_A(\Omega^l(M),\Omega^l(M)) \neq 0$, using \ref{propohelp}.
But by \ref{propohelp}, we also have $\Ext_A^t(\Omega^l(M),\Omega^l(M)) \cong \underline{\Hom}_A(\Omega^l(\Omega^t(M)), \Omega^l(M)) \cong \underline{\Hom}_A(\Omega^t(M),M) \cong \Ext_A^t(M,M)$, where we used \ref{propohelp} (5). Thus $\Ext_A^t(M,M) \neq 0$, which is a contradiction to our assumption that $M$ satisfies $\Ext_A^i(M,M \oplus A) =0$ for $i>0$. Thus no such non-projective $M$ can exist.

\end{proof} 

In particular, this proves that the Auslander-Reiten conjecture holds for monomial algebras since we saw that monomial algebras are weakly Gorenstein in \ref{corollarymonrefin}.

In the next subsection we propose a new homological conjecture and relate it to the Tachikawa conjecture.
\subsection{A new homological conjecture}
We have the following easy lemma:
\begin{lemma} \label{first lemma}
\begin{enumerate}
\item Let $A$ be an algebra of finite finitistic dimension, then every non-projective module $M$ with infinite dominant dimension has infinite projective dimension.
\item Let $A$ be an algebra of finite finitistic codominant dimension, then every non-projective module $M$ with infinite dominant dimension has infinite codominant dimension.
\end{enumerate}
\end{lemma}
\begin{proof}
\begin{enumerate}
\item Assume there is a non-projective module $M$ with finite projective dimension and infinite dominant dimension. Then the modules $\Omega^{-i}(M)$ have arbitrary large finite projective dimension and thus the finitistic dimension is infinite.
\item Assume there is a non-projective module $M$ with finite codominant dimension and infinite dominant dimension. Then the modules $\Omega^{-i}(M)$ have arbitrary large finite codominant dimension and thus the finitistic codominant dimension is infinite.
\end{enumerate}
\end{proof}
Note that in case one algebra $A$ has infinite finitistic dominant dimension then the opposite algebra has infinite finitistic codominant dimension.

The previous lemma motivates to study modules of infinite dominant dimension and look at their minimal projective resolution in order to test algebras for finite finitistic dimension and finite finitistic codominant dimension.
As an analog to the finitistic dimension conjecture, one can ask wheter the finitistic dominant dimension is always finite. We will give a negative answer in the last section using the previous lemma.

The next definition is taken from \cite{CIM}, where such algebras are used to give a new characterisation of representation-finite hereditary algebras.
\begin{definition}
Let $A$ be an algebra, then define the SGC-extension algebra of $A$ to be the algebra $\End_A(D(A) \oplus A)$, that is obtained from $A$ by taking the endomorphism ring of the smallest generator-cogenerator.
\end{definition}

\begin{proposition}
Consider the following four statements:
\begin{enumerate}
\item (Finitistic dimension conjecture) Any algebra has finite finitistic dimension.
\item Any non-projective module with infinite dominant dimension has infinite projective dimension.
\item Any non-injective semi-Gorenstein-injective module has infinite $\add(A \oplus D(A))$-resolution dimension.
\item (First Tachikawa conjecture) For any non-selfinjective algebra, there is an $i \geq 1$ with $\Ext_A^i(D(A),A) \neq 0$.
\end{enumerate}
We have $(1) \implies (2) \implies (3) \implies (4)$.
\end{proposition}
\begin{proof}
We saw $(1) \implies (2)$ already in \ref{first lemma}.
Assume now (2) and let $M$ be a semi-Gorenstein-injective module with finite $\add(A \oplus D(A))$-resolution dimension. The defining condition $\Ext_A^i(D(A),M)=0$ for all $i \geq 1$ gives that the module $\Hom_A(D(A) \oplus A, M)$ has infinite dominant dimension over the SGC-extension algebra $B=\End_A(A \oplus D(A))$, by \ref{mueller}.
Applying the functor $\Hom_A(A \oplus D(A),-)$ to a finite $\add(A \oplus D(A))$-resolution of $M$, one obtains that the $B$-module $\Hom_A(D(A) \oplus A, M)$ has finite projective dimension, contradicting (2).
Now assume (3) and $\Ext_A^i(D(A),A)= 0$ for all $i \geq 1$ in a non-selfinjective algebra. Then the module $M=A$ is a semi-Gorenstein-injective module with finite $\add(A \oplus D(A))$-resolution dimension, contradicting (3).
\end{proof}

The previous proposition motivates us to state (2) and (3) as new homological conjectures that lie between the classical finitistic dimension conjecture and the first Tachikawa conjecture:
\begin{conjecture}
\begin{enumerate}
\item Any non-projective module with infinite dominant dimension has infinite projective dimension.
\item Any non-injective semi-Gorenstein-injective module has infinite $add(A \oplus D(A))$-resolution dimension.
\end{enumerate}
\end{conjecture}
Those conjectures motivate us to study modules of infinite dominant dimension and semi-Gorenstein-injective modules.

We prove the conjecture (2) for Gorenstein algebras:
\begin{proposition}
Let $A$ be a Gorenstein algebra, then any non-injective semi-Gorenstein-injective module has infinite $\add(A \oplus D(A))$-resolution dimension.
\end{proposition}
\begin{proof}
First assume $A$ has Gorenstein dimension equal to zero, which is equivalent to $A$ being selfinjective. In this case $\add(A \oplus D(A))$-resolutions coincide with projective resolutions and the result follows since any non-injective module over a selfinjective algebra has infinite projective dimension.
Assume $A$ is a non-selfinjective Gorenstein algebra for the rest of the proof.
We use three facts about Gorenstein algebras:
\begin{enumerate}
\item An $A$-module $N$ has finite injective dimension if and only if it has finite projective dimension.
\item $A$ is weakly Gorenstein.

\item A non-projective Gorenstein-projective module has infinite projective dimension and a non-injective Gorenstein-injective module has infinite injective dimension.
\end{enumerate}
Assume $N$ is a non-injective semi-Gorenstein-injective module with finite $\add(A \oplus D(A))$-resolution dimension. Then there is an $\add(A \oplus D(A))$-resolution of the form:
$0 \rightarrow M_n \xrightarrow{f_n} \cdots M_0 \xrightarrow{f_0} N \rightarrow 0$ with $M_n \in \add(A \oplus D(A))$.
Let $K_r$ denote the kernel of $f_{r-1}$. Note that all $M_i$ have finite projective dimension, since $A$ is Gorenstein and thus $D(A)$ has finite projective dimension.
Then there is a short exact sequence $0 \rightarrow M_n \rightarrow M_{n-1} \rightarrow K_{n-1} \rightarrow 0$, which shows that also $K_{n-1}$ has finite projective dimension. Using induction, we see that also $N$ has finite projective dimension. But since $A$ is Gorenstein, the semi-Gorenstein-injective modules coincide with the Gorenstein-injective modules, which always have infinite injective dimension when they are non-injective. But for Gorenstein algebras, a module has infinite injective dimension if and only if it has infinite projective dimension and thus $N$ also has infinite projective dimension. This is a contradiction and thus $N$ has infinite $\add(A \oplus D(A))$-resolution dimension.

\end{proof}

\section{Gorenstein homological algebra of gendo-symmetric algebras}
The next propostion will be one of our main tools in this section. We recall the following from \cite{FanKoe2}, which gives a first clue for a connection between dominant dimensions and Gorenstein homological properties in gendo-symmetric algebras:
\begin{proposition}\label{domdimform}
Let $A$ be a non-selfinjective gendo-symmetric algebra and $M$ an $A$-module.
\begin{enumerate}
\item $M$ has dominant dimension larger than or equal to 2 if and only if $\nu^{-1}(M) \cong M$ and in this case the dominant dimension of $M$ is equal to $\inf \{ i \geq 1 | \Ext_A^{i}(D(A),M) \neq 0 \}$+1.
\item $M$ has codominant dimension larger than or equal to 2 if and only if $\nu(M) \cong M$ and in this case the codominant dimension of $M$ is equal to $\inf \{ i \geq 1 | \Ext_A^{i}(M,A) \neq 0 \}$+1.
\end{enumerate}
\end{proposition}
\begin{proof}
\begin{enumerate}
\item This is proven in \cite{FanKoe2} in Proposition 3.3.
\item This is dual to (1).
\end{enumerate}
\end{proof}
\begin{proposition}
Let $A$ be a finite dimensional algebra and $M$ an $A$-module. The following holds:
\begin{enumerate}
\item Let $P_1 \rightarrow P_0 \rightarrow M \rightarrow 0$ be a minimal projective presentation of $M$. Then there exists an exact sequence in $\mod-A$ of the form: \newline
$0 \rightarrow \tau(M) \rightarrow \nu(P_1) \rightarrow \nu(P_0) \rightarrow \nu(M) \rightarrow 0.$
\item Let $0 \rightarrow M \rightarrow I_0 \rightarrow I_1$ be a minimal injective corepresentation of $M$. Then there exists an exact sequence in $\mod-A$ of the form: \newline
$0 \rightarrow \nu^{-1}(M) \rightarrow \nu^{-1}(I_0) \rightarrow \nu^{-1}(I_1) \rightarrow \tau^{-1}(M) \rightarrow 0.$
\end{enumerate}
\begin{proof}
This follows from the definitions of Auslander-Reiten translates, see \cite{SkoYam} Chapter III, Proposition 5.3. 
\end{proof}
\end{proposition}
The following proposition is inspired by \cite{FanKoe3} Lemma 3.4:
\begin{proposition}\label{tauomega}
Let $A$ be a gendo-symmetric algebra and $M$ an $A$-module.
\begin{enumerate}
\item If $M$ has codominant dimension larger than or equal to 2 if and only if $\tau(M) \cong \Omega^{2}(M)$.
\item If $M$ has dominant dimension larger than or equal to 2 if and only if $\tau^{-1}(M) \cong \Omega^{-2}(M)$.
\end{enumerate}
\end{proposition}
\begin{proof}
We prove only (1), since the proof of (2) is dual.
We will often use that a module $M$ has codominant dimension at least 2 if and only if $\nu(M) \cong M$ by \ref{domdimform}.
Assume that $M$ has codominant dimension larger than or equal to 2. Assume that $P_1 \rightarrow P_0 \rightarrow M \rightarrow 0$ is the minimal projective presentation of $M$. Then by the previous proposition, there is the following exact sequence: \newline
$0 \rightarrow \tau(M) \rightarrow \nu(P_1) \rightarrow \nu(P_0) \rightarrow \nu(M) \rightarrow 0.$
But since $P_1,P_0$ are also injective and thus have codominant dimension larger than or equal to 2
$\nu(P_1) \cong P_1$ and $\nu(P_0) \cong P_0$. As $M$ also has codominant dimension larger than or equal to 2: $\nu(M) \cong M$ and the exact sequence looks like the beginning of a minimal projective resolution of $M$: \newline
$0 \rightarrow \tau(M) \rightarrow P_1 \rightarrow P_0 \rightarrow M \rightarrow 0.$
Thus $\Omega^{2}(M) \cong \tau(M)$. \newline
Assume now that $\Omega^{2}(M) \cong \tau(M)$. Then look at the minimal projective presentation $P_1 \rightarrow P_0 \rightarrow M \rightarrow 0$ of $M$ and by the previous proposition we get a minimal injective coresolution of $\tau(M)$ as follows. $0 \rightarrow \tau(M) \rightarrow \nu(P_1) \rightarrow \nu(P_0) \rightarrow \nu(M) \rightarrow 0$. But now since $A$ is gendo-symmetric and thus has dominant dimension at least two, we have $\Omega^{2}(\mod-A)=\Dom_2(A)$ by \ref{torsionless} (1) and thus $\tau(M) \cong \Omega^{2}(M)$ has dominant dimension at least two and thus in the above minimal injective coresolution $\nu(P_1)$ and $\nu(P_0)$ are projective-injective, which implies that $P_0$ and $P_1$ are also projective-injective, since in a gendo-symmetric algebra a module $P$ is projective-injective if and only if $P \in \add(eA)$. Thus $M$ has codominant dimension at least two.
\end{proof}
We also refer to \cite{Mar4}, for a characterisation of modules $M$ with the similar property $\tau(M) \cong \nu \Omega^2(M)$ for general Artin algebras.
The following theorem gives the first link between dominant dimensions of modules and Gorenstein homological algebra:
\begin{theorem}
\label{mainresult}
Let $A$ be a gendo-symmetric algebra and $M$ an $A$-module.
Then the following are equivalent:
\begin{enumerate}
\item $M$ is Gorenstein-projective-injective.
\item $M$ has infinite dominant dimension and infinite codominant dimension.
\end{enumerate}
\end{theorem}
\begin{proof}
$(1) \Rightarrow (2)$: A projective module is projective-injective if and only if it is a Gorenstein-projective-injective module.
Assume now that $M$ is Gorenstein-projective-injective and not projective. We conclude using \ref{torsionless} (1) and its dual that $M$ has dominant dimension and codominant dimension larger than or equal to 2, since $A$ has dominant dimension at least 2.
Since $M$ is Gorenstein-injective, $\Ext_A^{i}(D(A),M)=0$ for all $i \geq 1$ and thus $M$ has infinite dominant dimension by \ref{domdimform}. Since $M$ is Gorenstein-projective $\Ext_A^{i}(M,A)=0$ for all $i \geq 1$ and thus $M$ has infinite codominant dimension again by \ref{domdimform}. \newline
$(2) \Rightarrow (1)$:
Assume now that $M$ has infinite dominant and infinite codominant dimension. Clearly then also all the modules $\Omega^{i}(M)$ have infinite dominant and codominant dimensions for every $i \in \mathbb{Z}$. Then $\nu^{-1}(M) \cong M$, since the dominant dimension of $M$ is larger than or equal 2. And dually $\nu(M) \cong M$, since the codominant dimension of $M$ is larger than or equal to 2.
Since $M$ has infinite codominant dimension, we get $\Ext_A^{i}(M,A)=0$ for all $ i \geq 1$ and with the previous proposition \ref{tauomega} the following holds: $\tau(M) \cong \Omega^{2}(M)$. Now $M$ is Gorenstein-projective if and only if additionally $\Ext_A^{i}(\Tr(M),A)=0$ for all $i \geq 1$. But $\Ext_A^{i}(\Tr(M),A)= \Ext_A^{i}(D(A),\tau(M))=\Ext_A^{i}(D(A),\Omega^{2}(M)))=0$ for all $i \geq 1$, since with $M$ also $\Omega^{2}(M)$ has infinite dominant dimension.  Thus $M$ is Gorenstein-projective.
Since $M$ has infinite dominant dimension, $\Ext_A^{i}(D(A),M)=0$ for all $ i \geq 1$ and by the previous lemma the following holds: $\tau^{-1}(M) \cong \Omega^{-2}(M)$. Now $M$ is Gorenstein-injective if and only if additionally $\Ext_A^{i}(\tau^{-1}(M),A)=0$ for all $i \geq 1$. But $\Ext_A^{i}(\tau^{-1}(M),A)=\Ext_A^{i}(\Omega^{-2}(M),A)=0$ for all $i \geq 1$, since with $M$ also $\Omega^{-2}(M)$ has infinite codominant dimension. 
\end{proof}

\begin{corollary}
\label{CorCM}
Let $A$ be a gendo-symmetric non-selfinjective algebra.
\begin{enumerate}
\item If an Auslander-Reiten component contains a non-projective-injective Gorenstein-projective-injective module then this component consists only of Gorenstein-projective-injective modules. Thus the indecomposable non-projective-injective Gorenstein-projective-injective modules form unions of stable Auslander-Reiten components of the algebra.
\item If $A$ is CM-finite, then there is no non-projective-injective Gorenstein-projective-injective $A$-module.
\end{enumerate}
\end{corollary}
\begin{proof}
\begin{enumerate}
\item Assume that $M$ is a non-projective-injective Gorenstein-projective-injective module. With $M$ being Gorenstein-projective-injective also $\tau(M) \cong \Omega^{2}(M)$ and $\tau^{-1}(M) \cong \Omega^{-2}(M)$ (using \ref{tauomega}) are Gorenstein-projective-injective, since being Gorenstein-projective-injective is equivalent to having infinite dominant and infinite codominant dimension by the previous theorem. 
In a Auslander-Reiten sequence $ 0 \rightarrow \tau(M) \rightarrow S \rightarrow M \rightarrow 0$, also $S$ is Gorenstein-projective-injective, since with $M$ and $\tau(M) \cong \Omega^{2}(M)$ also the middle term $S$ has infinite dominant and codominant dimension by the Horseshoe lemma.
Thus every module in the Auslander-Reiten component containing $M$ is a Gorenstein-projective-injective module.
\item Assume that $A$ is CM-finite and contains a non-projective-injective module $M$ that is Gorenstein-projective-injective. Then $A$ contains a whole Auslander-Reiten component of modules that are Gorenstein-projective-injective by 1. Then there are 2 possible cases:
\begin{enumerate}
\item The Auslander-Reiten component is infinite. But then $A$ is not CM-finite. This is a contradiction.
\item The Auslander-Reiten component is finite. Then $A$ is representation finite and thus the Nakayama conjecture holds true for $A$. But by (1), every module in the Auslander-Reiten component has infinite dominant and codominant dimension despite that fact that the dominant dimension of $A$ is finite. This is a contradiction.
\end{enumerate}
\end{enumerate}
\end{proof}

We have the following diagram for gendo-symmetric algebras $A$ with minimal faithful projective-injective module $eA$, where we use the notations as in \hyperref[ARSmaintheorem]{ \ref*{ARSmaintheorem}} and $\mathcal{W}:= \{ X \in ^{\perp}{Ae} \cap {Ae}^{\perp}| G_1(X) \cong G_2(X) \}$ (we will see later that $\mathcal{W}$ can be described much easier if the algebra is weakly Gorenstein): \newline
\underline{diagram 1}
$$\xymatrix@1{ & \text{Gpi}(A)=\text{Dom(A)}\cap \text{Codom(A)} \mystrut\ar@{^{(}->}[dr] \ar@{_{(}->}[dl] \ar[dddd]_{(-)e} &  \\ \text{Dom(A)} \ar@/^/[dd]^{F_2=(-)e}  &  & \text{Codom(A)} \ar@/^/[dd]^{F_1=(-)e} \\  &  &  \\ {(Ae)}^{\perp} \ar@/^/[uu]^{G_2=\text{Hom}_{eAe}(eA,-)}&  & {}^{\perp}(Ae) \ar@/^/[uu]^{G_1=(-)\otimes_{eAe} eA} \\ & \mathcal{W} \mystrut\ar@{_{(}->}[ru]\ar@{^{(}->}[lu] & }$$
Using this diagram, we can prove the next theorem:

\begin{theorem}
\label{diagram}
Let $A$ be a gendo-symmetric non-selfinjective algebra with minimal faithful projective-injective module $eA$. Then $(-)e: \Gpi(A) \rightarrow \mathcal{W}$ is an equivalence.
\end{theorem}
\begin{proof}
Recall that $F_i$ is an equivalence with quasi-inverse $G_i$ as explained in \ref{ARSmaintheorem} for $i=1,2$. This makes it clear that $(-)e$ really maps $\Gpi(A)$ to $\mathcal{W}$. Since $(-)e$ is fully faithful on $\Dom(A)$, also the restiction to $\Gpi(A)$ is fully faithful. Now let $X \in \mathcal{W}$, then just note that $(G_i(X))e \cong F_i(G_i(X)) \cong X$ for $i=1,2$ and thus $(-)e$ is dense since the module $G_i(X)$ is in $\Gpi(A)$ by definition and (4) of \ref{ARSmaintheorem}.
\end{proof}

\begin{theorem}
\label{finitistic}
Let $A$ be a non-selfinjective gendo-symmetric algebra, which additionally is a weakly Gorenstein algebra.
\begin{enumerate}
\item
The following are equivalent for a non-injective and non-projective indecomposable module $M$:
\begin{enumerate}
\item[i)] $M$ has infinite dominant dimension.
\item[ii)] $M$ is Gorenstein-projective-injective and $\nu^{-1}(M) \cong M$.
\item[iii)] $M$ is Gorenstein-projective-injective.
\item[iv)]  $M$ has infinite codominant dimension.
\end{enumerate}
\item If $A$ is CM-finite and Gorenstein with Gorenstein dimension $g$, then every non-injective module has finite dominant dimension and $\fdomdim(A) \leq g+1$.
\end{enumerate}
\end{theorem}
\begin{proof}
\begin{enumerate}
\item $i) \Rightarrow ii)$: Assume $M$ has infinite dominant dimension. Since $A$ is a gendo-symmetric algebra, this is equivalent to the two conditions $\nu^{-1}(M)=\Hom_A(D(A),M) \cong M$ and $\Ext_A^{i}(D(A),M)=0$ for all $i\geq 1$. Since we assume that the subcategory of Gorenstein-injective modules coincides with $D(A)^{\perp}$, $M$ is Gorenstein-injective. Now by \ref{taugorenstein} (1), if a module $X$ is Gorenstein-injective, then $\Hom_A(D(A),X)$ is Gorenstein-projective. Using this with $X=M$, we get that $M \cong \Hom_A(D(A),M)$ is Gorenstein-projective. \newline
$ii) \Rightarrow i)$: Since $\nu^{-1}(M)=M$, $M$ has dominant dimension larger than or equal to 2. Because $M$ is Gorenstein-injective, $\Ext_A^{i}(D(A),M)=0$ for all $i\geq 1$. Then $\domdim M=\infty$ is clear. \newline
$ii) \Leftrightarrow iii)$: Since $A$ is gendo-symmetric, $A$ has dominant dimension at least two and using \ref{torsionless} (1), we see that all Gorenstein-projective modules have dominant dimension at least 2. But in a gendo-symmetric algebra having dominant dimension at least two is equivalent to $\nu^{-1}(M) \cong M$. \newline
$iv) \Leftrightarrow iii)$: Now $\codomdim M= \infty$ is equivalent to $\domdim D(M)=\infty$ and using the equivalence of $i)$ and $ii)$, this is equivalent to $\nu^{-1}(D(M)) \cong D(M)$ and $D(M)$ is Gorenstein-injective and Gorenstein-projective. But, since $\nu^{-1}=\Hom_A(-,A) \circ D$, this is equivalent to $\Hom_A(M,A) \cong D(M)$ and that $M$ is Gorenstein-projective and Gorenstein-injective.
Now $\nu^{-1}(M)=DD(M) \cong M$ and the equivalence of $iii)$ and $iv)$ is clear.
\item By \hyperref[CorCM]{ \ref*{CorCM}} , $A$ does not contain a non-projective Gorenstein-projective-injective module and thus no non-injective module of infinite dominant dimension.
Assume now $M$ is a module with $\infty>i=\domdim M > g+1$. Then $\Ext_A^{i-1}(D(A),M) \neq 0$ by \ref{domdimform}, contradicting the fact that $D(A)$ has projective dimension equal to $g$. 
\end{enumerate}

\end{proof}

The next example shows that the bound for the finitistic dominant dimension in the previous theorem is optimal:
\begin{example}
Let $A$ be the so called penny-farthing algebra (see \cite{GR}) with two simple modules given by quiver and relations as follows: 
$$
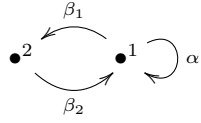
\begin{xy}
  \xymatrix{
      & {\bullet}^2 \ar@/ _1pc/[r]_{\beta_2}   & {\bullet}^1 \ar@(ur,dr)[]^{\alpha} \ar@/ _1pc/[l]_{\beta_1}}
\end{xy}
$$
The relations are: $I=<\alpha^2-\beta_1 \beta_2, \beta_2 \beta_1>.$ Thus $A=kQ/I$ and $A$ is a symmetric algebra.
Let $S_i$($e_i$) be the simple module (primitive idempotent) corresponding to the vertex $i$ and $J$ the radical of $A$. We show that the algebra $B:=\End_A(A \oplus S_2)$ has dominant dimension and Gorenstein dimension equal to 3 and finitistic dominant dimension equal to 4. Note that $B$ is a CM-finite weakly Gorenstein algebra using \ref{corollarymonrefin} (2), since the penny-farthing algebra is representation-finite. By explicitly giving the relevant minimal projective resolutions (see below) and using \ref{benson}, we will show that $\Ext_A^{1}(S_2,S_2)=0$ but $\Ext_A^{2}(S_2,S_2) \neq 0$ and $\Ext_A^{i}(S_2,e_2J^2)=0$ for $i=1$ and $i=2$, but $\Ext_A^{3}(S_2,e_2J^2) \neq 0$. Thus the gendo-symmetric algebra $B$ has dominant dimension 3 and the $B$-module $\Hom_A(A \oplus S_2, e_2 J^{2})$ has dominant dimension 4 by \hyperref[ARSmaintheorem]{ \ref*{ARSmaintheorem}} and \hyperref[benson]{ \ref*{benson}}. Because of $\tau(\Omega^{1}(S_2))=\Omega^3(S_2)=S_2$, one also conlcudes that $B$ has Gorenstein dimension 3 using \hyperref[theoremchekoe]{\ref*{theoremchekoe}}. The relevant minimal projetive resolutions of $S_2$ and $e_2J^{2}$ needed to proof the above statements are (note that both modules have period 3 and thus one can read off minimal projective and minimal injective resolutions):
$$
\begin{xy}
	\xymatrix@C-1pc@R-1pc{
	& e_2A\ar[rd]^{}   &         & e_2A \ar[rd]^{} &     & e_1A \ar[rd]^{}  &   & e_2A\ar[r] & S_2\ar[r]^{}  &  0 \\
  \cdots\ar[ru]^{} &     & S_2 \ar[ru]^{} &    & \beta_1 A \ar[ru]^{} &    & e_2J^{1}\ar[ru]^{} &  &  & }
\end{xy}
$$ \newline
$$
\begin{xy}
	\xymatrix@C-1pc@R-1pc{
	& e_1A \ar[rd]^{}   &         & e_2A \ar[rd]^{} &     & e_1A \ar[rd]^{}  &   & e_1A \ar[r] & e_2J^{2} \ar[r]^{}  &  0 \\
  \cdots\ar[ru]^{} &     & e_2J^{2} \ar[ru]^{} &    & \alpha \beta_1A \ar[ru]^{} &    & \alpha A \ar[ru]^{} &  &  & }
\end{xy}
$$

\end{example}

\begin{corollary}
\label{perpcor}
Let $A$ be a gendo-symmetric and weakly Gorenstein algebra with minimal faithful projective-injective module $eA$.
Then the functor $(-)e : \Gpi(A) \rightarrow eA^{\perp} \subseteq \mod-eAe$ is an equivalence of categories and $eA^{\perp} = ^{\perp}{Ae} = ^{\perp}{Ae} \cap {Ae}^{\perp}$.
\end{corollary}
\begin{proof}
This follows from \hyperref[diagram]{ \ref*{diagram}} and \hyperref[finitistic]{ \ref*{finitistic}}. 
\end{proof}

The following proposition allows us to check whether a module over a gendo-symmetric Gorenstein algebra is Gorenstein-projective-injective by calculating only finitely many terms in a projective or injective minimal resolution of this module. 
\begin{proposition}
Let $A$ be a gendo-symmetric algebra of Gorenstein dimension $g$. Then for a module $M$ the following are equivalent:
\begin{enumerate}
\item $M$ is Gorenstein-projective-injective.
\item $\domdim M+\codomdim M \geq 2g$.
\item $\domdim M \geq 2g$.
\item $\codomdim M \geq 2g$.

\end{enumerate}
\end{proposition}
\begin{proof}
The proposition is clear when $A$ is selfinjective and thus we assume that $A$ is non-selfinjective in the following.
$(1) \Rightarrow (2):$ \newline
We saw in \hyperref[mainresult]{ \ref*{mainresult}}  that a module $M \in \Gpi(A)$ has infinite dominant and codominant dimension, thus $\domdim M+\codomdim M \geq 2g$. \newline
$(2) \Rightarrow (1):$ \newline
Assume $\domdim M=i$ and $\codomdim M=j$, and $i+j \geq 2g$. Without loss of generality, assume that $i \geq j$. Then the module $W=\Omega^{-i+g}(M)$ can be written $W=\Omega^{g}(\Omega^{-i}(M))$ and $W=\Omega^{-g}(\Omega^{s}(M))$ for some $s \leq j$. Thus by the characterisation of Gorenstein algebras \hyperref[gorensteinkrit]{ \ref*{gorensteinkrit}}, $W$ is Gorenstein-projective-injective and so is every syzygy of $W$ and thus also $M=\Omega^{-k}(W)$. \newline
$(1) \Rightarrow (3):$ \newline
This is clear since a Gorenstein-projective-injective module has infinite dominant dimension in a gendo-symmetric algebra by \ref{mainresult}. \newline
$(3) \Rightarrow (2):$ \newline
This is clear. \newline
$(4) \Leftrightarrow (1) $ is dual to $3. \Leftrightarrow 1.$ 
\end{proof}

We now give a class of examples of non-selfinjective gendo-symmetric algebras containing Auslander-Reiten compontents of Gorenstein-projective-injective modules.

\begin{proposition}
Let $B$ be a symmetric algebra and $M:=B \oplus W$, where $W$ is a nonzero 2-periodic module.
\begin{enumerate}
\item The gendo-symmetric algebra $A:=\End_B(M)$ has dominant dimension 2 and Gorenstein dimension 2.
\item $^{\perp}M= M^{\perp}$
\item $M^{\perp}=\{X \in \mod-B | \Ext_B^{i}(M,X)=0$, for $i=1,2 \}$.
\item If $W$ is even 1-periodic, then  $M^{\perp}=\{X \in \mod-B | \Ext_B^{1}(M,X)=0 \} \}$.
\end{enumerate}
\end{proposition}
\begin{proof}
\begin{enumerate}
\item We use  \hyperref[extrechnen]{ \ref*{extrechnen}}: $\Ext_B^{1}(M,M) \cong \underline{\Hom}_B(M,\Omega^{2}(M)) \cong \underline{\Hom}_B(M,M) \neq 0$, since the identity does not factor over the projectives. Thus $A$ has dominant dimension 2 by Mueller's theorem \ref{mueller}. Now we use \hyperref[theoremchekoe]{ \ref*{theoremchekoe}} to calculate the Gorenstein dimension. But the right $\add(M)$-approximitation of $\tau(M) \cong \Omega^{2}(M) \cong M$ is an isomorphism $f:\Omega^{2}(M) \rightarrow M$, and thus the length of a resolution is 0 and so the right Gorenstein dimension is also equal to two. Similar the left Gorenstein dimension is calculated to be 2.
\item This follows by \hyperref[perpcor]{ \ref*{perpcor}}.
\item We can write every natural number $i$ as $i=2r+p,$ for $r \geq 0$ and $p \in \{1,2\}$. Then $\Ext_B^{i}(M,X) \cong \underline{\Hom}_B(\Omega^{i}(M),X) \cong \underline{\Hom}_B(\Omega^{p}(\Omega^{2r}(M),X) \cong \Ext_B^{p}(M,X)$, since $M$ is 2-periodic.
\item In this case $\Ext_B^{i}(M,X) \cong \underline{\Hom}_B(\Omega^{i}(M),X) \cong \underline{\Hom}_B(\Omega^{1}(M),X) \cong \Ext_B^{1}(M,X) $ for every $i \geq 1.$
\end{enumerate}
\end{proof}

\begin{remark}
In a tame algebra, all but finitely many indecomposable modules $M$ of a given dimension have the property that $\tau(M) \cong M$, by a famous result of Crawley-Boevey, see \cite{Cra}. Thus with the previous proposition one might construct alot of examples, since in a tame symmetric algebra $\tau(M) \cong \Omega^{2}(M) \cong M$ then holds  for all but finitely many indecomposable modules of a given dimension.
\end{remark}

\begin{example}
Assume $K$ is an infinite field of characteristic 2. \newline
Let $A$ be the symmetric algebra $K[x,y]/(x^2,y^2,xy-yx)$ (which is isomorphic to the group algebra of the Klein four group).
We use the notation and results as in Example 10.7 in page 417 of \cite{SkoYam}.
There $M(a,b)$ is defined as the module $A/(ax+by)A$, for $a,b$ nonzero elements of $K$.
Note that $M(a,b)$ is isomorphic to $M(c,d)$ for nonzero $c,d$ if and only if there is an $l \in k$ with $c=l \cdot a$ and $d=l \cdot b$.
Then $\Omega^{1}(M(a,b))=(ax+by)A \cong M(-a,b)$ (note this holds despite the fact that they excluded $\lambda=1$ in \cite{SkoYam}).
But $M(-a,b)=M(a,b)$ since we assume that the field has characteristic 2 and thus $M(a,b)$ is 1 periodic and we can apply the previous proposition.
We now want to calculate $\Ext_A^{1}(M(a,b),M(c,d))$, for nonzero $c$ and $d$ such that $M(c,d)$ is not isomorphic to $M(a,b)$.
For this we use the short exact sequence $0 \rightarrow M(a,b) \rightarrow A \rightarrow M(a,b) \rightarrow 0$ and apply the functor $\Hom_A(-,M(c,d))$ to get the long exact sequence: \newline
$0 \rightarrow \Hom_A(M(a,b),M(c,d)) \rightarrow \Hom_A(A,M(c,d)) \rightarrow \Hom_A(M(a,b),M(c,d)) \newline \rightarrow \Ext_A^{1}(M(a,b),M(c,d)) \rightarrow \cdots $. \newline
Now we see that $\Ext_A^{1}(M(a,b),M(c,d))=0$ if and only if \newline $0 \rightarrow \Hom_A(M(a,b),M(c,d)) \rightarrow \Hom_A(A,M(c,d)) \rightarrow \Hom_A(M(a,b),M(c,d)) \rightarrow 0$ is a short exact sequence, which is the case if and only if $2=\dim(M(c,d))=\dim(\Hom_A(A,M(c,d)))=2 \dim(\Hom_A(M(a,b),M(c,d))).$
But the last equation is true because $\Hom_A(M(a,b),M(c,d))=1$, since up to multiplication by a scalar the only homomorphism is the projection from the top into the socle (note that both modules have dimension $2$). \newline Thus $\Ext_A^{1}(M(a,b),M(c,d))=0$ and $\Ext_A^{k}(M(a,b),M(c,d))=0$ for all $k \geq 1$ since \newline $\Ext_A^{k}(M(a,b),M(c,d))=\Ext_A^{1}(\Omega^{k-1}(M(a,b)),M(c,d))=\Ext_A^{1}(M(a,b),M(c,d))=0$.
Then $B:=\End_A(A\oplus M(a,b))$ is a Gorenstein and gendo-symmetric algebra.
Thus by the previous proposition, the algebra $B$ has $\Hom_A(A \oplus M(a,b),M(c,d))$ as a Gorenstein-projective-injective module and by \hyperref[CorCM]{ \ref*{CorCM}} it has a whole Auslander-Reiten component consisting of Gorenstein-projective-injective modules.

\end{example}

We note that most of the previous results are special to gendo-symmetric algebras as the following example shows:

\begin{example}
Take the CNakayama algebra $A$ with Kupisch series $(3s+1,3s+2,3s+2), s \geq 1$ ($A$ is not gendo-symmetric). The dominant dimension and the Gorenstein dimension of $A$ are both 2.
The finitistic dominant dimension equals 4, while the finitistic dimension equals the Gorenstein dimension which is 2 (note that by  \hyperref[finitistic]{ \ref*{finitistic}}  (2), this can not happen for gendo-symmetic CM-finite Gorenstein algebras).
The modules of the form $e_1 A/e_1 J^k$ with $k \equiv 0$ are exactly the Gorenstein-projective-injective modules but those modules have dominant dimensions equal to 2 and not infinite when they are not projective-injective. There is also no Auslander-Reiten component consisting only of Gorenstein-projective-injective modules despite the existence of a non-projective-injective Gorenstein-projective-injective indecomposable module.

\end{example}

\section{The Liu-Schulz algebra and non-weakly Gorenstein algebras}
We first need the following result in order to calculate the codominant dimension of algebras:
\begin{proposition} \label{codomdimprop}
Let $A$ be a selfinjective algebra and $M=A \oplus X$ a generator with $X$ having no projective direct summands. Let $B:=\End_A(M)$ and $Y:=\Hom_A(M,N)$ for an $A$-module $N$.
\begin{enumerate}
\item The codominant dimension of the $B$-module $Y$ equals $\inf \{ i \geq 1 | \Ext_A^{1}(M,\Omega^{i}(N)) \neq 0 \} -1.$
\item $Y$ has infinite dominant and infinite codominant dimension if and only if $\Ext_A^{1}(M, \Omega^{i}(N))=0$ for all $i \in \mathbb{Z}$.
\end{enumerate}
\end{proposition}
\begin{proof}
Note that minimal $\add(M)$-approximations of an $A$-module $N$ correspond to the projective cover of the $B$-module $\Hom_A(M,N)$ by applying the functor $\Hom_A(M,-)$. Let $R$ be the direct sum of all projective-injective indecomposable $B$-modules. Then one has $\add(R)=\add(\Hom_A(M,A))$. Thus the codominant dimension of the module $Y$ is at least one iff its projective cover is in $\add(R)$ iff the minimal $\add(M)$-approximation of $N$ is equal to the projective cover of $N$ in $\mod-A$.
Let $f: P \rightarrow N$ be the projective cover of $N$ in $\mod-A$. Then this is a minimal $\add(M)$-approximation if and only if $\Hom_A(M,f)$ is surjective.
Applying to the short exact sequence $0 \rightarrow \Omega^1(N) \rightarrow P \rightarrow N \rightarrow 0$ the functor $\Hom_A(M,-)$ we get the exact sequence:
$0 \rightarrow \Hom_A(M,\Omega^1(N)) \rightarrow \Hom_A(M,P) \rightarrow \Hom_A(M,N) \rightarrow \Ext_A^1(M,\Omega^1(N)) \rightarrow 0$.
This gives us that the projective cover is the minimal $\add(M)$-approximation if and only if $\Ext_A^1(M,\Omega^1(N))=0$. We obtain (1) by applying the same argument to $\Omega^l(N)$ for $l \geq 1$.
(2) is a consequence of (1) in combination with Mueller's theorem \ref{mueller}.
\end{proof}

We can now prove the main theorem of this section.
\begin{theorem} \label{maintheorem}
Let $A$ be a symmetric algebra and $X$ a direct sum of indecomposable non-projective modules. Let $M$ be an indecomposable module such that $\Ext_A^{l}(X,M) \neq 0$ for some $l \geq 1$ and $\Ext_A^i(X,M)=0$ for all $i \geq l+1$. 
Then the gendo-symmetric algebra $B:=\End_A(A \oplus X)$ is not weakly Gorenstein and has infinite finitistic codominant dimension.
\end{theorem}
\begin{proof}
Let $N:=A \oplus X$ and $B:=\End_A(N)$. Then $B$ is by definition gendo-symmetric.
We now look at the $B$-module $R:=\Hom_A(N,\Omega^{-l}(M))$. Since we have $\Ext_A^i(X,\Omega^{-l}(M))=\Ext_A^{i+l}(X,M)=0$ for all $i \geq 1$, by \ref{mueller} the module $R$ has infinite dominant dimension.
Now using \ref{codomdimprop}, $R$ has codominant dimension zero since $\Ext_A^1(X,\Omega^1(\Omega^{-l}(M)))=\Ext_A^1(X,\Omega^{-(l-1)}(M))=\Ext_A^{l}(X,M) \neq 0$.
Now one has by \ref{domdimform} (2) that $R$ is a semi-Gorenstein-injective module, since modules $W$ with infinite dominant dimension over gendo-symmetric algebras have $\Ext_A^i(D(A),W)=0$ for all $i \geq 1$. By (1) of \ref{torsionless} it can not be Gorenstein-injective since it has codominant dimension zero, while all Gorenstein-injective modules have codominant dimension at least two since they are second cosyzygy modules and $A$ has dominant dimension at least two.
Thus $B$ is not weakly Gorenstein, as it contains a module that is semi-Gorenstein-injective but not Gorenstein-injective.
Now look at the $B$-modules $\Omega^{-i}(R)$ for $i \geq 1$. Since $R$ has infinite dominant dimension and codominant dimension zero, those modules have codominant dimension equal to $i$. Thus the algebra $B$ has infinite finitistic codominant dimension.
\end{proof}
We remark that in the proof of the previous theorem we constructed a semi-Gorenstein-injective module $R$ over an algebra $B$ that is not Gorenstein-injective. To obtain semi-Gorenstein-projective modules that are not Gorenstein-projective, one can just look at $D(R)$ over the opposite algebra of $B$.
We now construct non-local algebras with semi-Gorenstein-injective modules that are not Gorenstein-injective with an arbitrary large number of simple modules.
Let $A$ be the Liu-Schulz algebra, see for example \cite{Rin4} for an article about this algebra.
That is $A=A_r:=K<x,y,z>/(x^2,y^2,z^2,yx+rxy,zy+ryz,xz+rzx>$ for some nonzero field element $r$. We often write short $A$ instead of $A_r$. For non-zero field elements $c$, we define the right $A$-modules $M_c:=A/(x+c y)A$. In the following $c,d,e,f$ will denote field elements. We assume that $r^2 \neq 1$ and $r^3 \neq 1$ for simplicity in the following.

We collect some results for this algebra and those modules, where we refer to \cite{Rin4} for proofs:
\begin{proposition}
Let $A$ be the Liu-Schulz algebra, then
\begin{enumerate}
\item $A$ is a local 8-dimensional symmetric algebra.
\item $M_c$ is 4-dimensional and $M_c$ is isomorphic to $M_d$ if and only if $c=d$.

\end{enumerate}
\end{proposition}
Note that $\Omega^{1}(M_c)=(x+cy)A \cong M_{cr}$.

\begin{lemma} \label{mainlemma}
Let $A$ be the Liu-Schulz algebra, then
\begin{enumerate}
\item $\dim(\Hom_A(M_c,M_e))=2+\delta_{c,e}+\delta_{c r^2,e}$, where $\delta_{s,t}$ is the Kroenecker delta.
\item $\Ext_A^{1}(M_c,M_d)=0$ if and only if $c \neq d, cr \neq d, cr^2 \neq d$ and $cr^3 \neq d$.
\item $\Ext_A^{i}(M_c,M_d)=0$ for all $i \geq 1$ if and only if $d \neq r^l c$ for all $l \geq 0$.
\end{enumerate}
\end{lemma}
\begin{proof}
\begin{enumerate}
\item We calculate $dim(\Hom_A(A/(x+cy)A, (x+dy)A)))$, which gives the result by setting $e:=dr$ since $(x+dy)A \cong A/(x+dry)A$.
The elements in $\Hom_A(A/(x+cy)A, (x+dy)A))$ are of the form $l_z$, which is the homomorphism by right multplication by an element $z \in (x+dy)A$ with $z (x+cy)A$=0, which is equivalent to $z (x+cy)=0$.
Now a general $z \in (x+dy)A$ can be written as $z=a_1 (x+dy) + a_2 xy + a_3 xyz+a_4(xz+dyz)$ and the condition $z (x+cy)=0$ leads to $a_1(x+dy)(x+cy)+a_4 (xz+dyz)(x+cy)$. This is simplified to $a_1(cxy-rd xy)+a_4(-rcyzx+dyzx)=0$. Now this gives that we can choose $a_2$ and $a_3$ arbitrary and $a_1=0$ iff $e \neq c$ and $a_4=0$ iff $e \neq r^2 c$. This gives the result.

\item Look at the short exact sequence:
$$0 \rightarrow A/(x+cry)A \rightarrow A \rightarrow A/(x+cy)A \rightarrow 0$$
and apply the functor $\Hom_A(-,A/(x+dy)A)$ to obtain the exact sequence: \newline
\begin{tiny}$$0 \rightarrow \Hom_A(A/(x+cy)A,A/(x+dy)A) \rightarrow \Hom_A(A,A/(x+dy)A) \rightarrow \Hom_A(A/(x+cry)A,A/(x+dy)A) \rightarrow \Ext_A^{1}(A/(x+cy)A,A/(x+dy)A) \rightarrow 0.$$ \end{tiny}
Thus, $\dim(\Ext_A^{1}(A/(x+cy)A,A/(x+dy)A))=\dim(\Hom_A(A/(x+cy)A,A/(x+dy)A))+\dim(\Hom_A(A/(x+cry)A,A/(x+dy)A))-\dim(\Hom_A(A,A/(x+dy)A))= 2+\delta_{c,d}+\delta_{cr^2,d}+2+\delta_{cr,d}+\delta_{cr^3,d}-4$ by (1).
This gives $\dim(\Ext_A^{1}(A/(x+cy)A,A/(x+dy)A))=0$ if and only if $c \neq d, cr \neq d, cr^2 \neq d$ and $cr^3 \neq d$.
\item This follows directly by (2) using that $\Ext_A^{i}(M_c,M_d)=\Ext_A^{1}(\Omega^{i-1}(M_c),M_d)=\Ext_A^{1}(M_{r^{i-1}c},M_d)$.
\end{enumerate}
\end{proof}

\begin{theorem} \label{theoreminfinitefindomdim}
Let $K$ be a field with an element $r$ such that $r^l \neq 1$ for all $l \geq 1$ and let $c_1, ... , c_n$ be $n$ pairwise distinct field elements. Let $N:=A_r \oplus M_{c_1} \oplus ... \oplus M_{c_n}$.

\begin{enumerate}
\item $B$ is a gendo-symmetric algebra of dominant dimension equal to two with $n+1$ simple modules.

\item The algebra $B:=\End_{A_r}(N)$ has infinite finitistic codominant dimension and is not weakly Gorenstein.
\end{enumerate}
\end{theorem}
\begin{proof}
Set $A=A_r$.
\begin{enumerate}
\item $B$ is by definition gendo-symmetric, since $A_r$ is symmetric. By \ref{mainlemma}, $\Ext_A^{1}(M_{c_i},M_{c_i}) \neq 0$ and thus $B$ has dominant dimension equal to two by Mueller's theorem. Since the module $N$ is basic, the number of simple modules of $B$ equals the number of indecomposable summands of $N$.
\item We want to use \ref{maintheorem} to prove the result and show that the assumptions of \ref{maintheorem} are satisfied.
Choose $Z:=M_{c_i}$ with $i$ choosen so that for $d:=\frac{c_i}{r}$, $d \neq r^l c_j$ for any $l \geq 0$ and any $j$. Note that $M_d \cong \Omega^{-1}(M_{c_i})$.
Then we have $\Ext_A^{1}(N,Z) \neq 0$ since $\Ext_A^{1}(Z,Z) \neq 0$ by \ref{mainlemma} (2). But for $l \geq 2$ we have $\Ext_A^{l}(N,Z)=\Ext_A^{l-1}(N,\Omega^{-1}(Z))=\Ext_A^{l-1}(N,M_d)$=0 by \ref{mainlemma} (2), since $\Ext_A^{i}(M_{c_j},M_d)=0$ for any $i \geq 1$ and any $j$ because $d \neq r^i c_j$ for any $i \geq 0$ and any $j$.

\end{enumerate}
\end{proof}

\section{Open Questions}
We note that Liu-Schulz algebras are a special case of quantum complete intersection algebras $K<x_1,x_2,...,x_n>/(x_i^{a_i}, x_i x_j + q_{i,j} x_j x_i )$ for $i >j$, $a_i \geq 2$ natural numbers and $q_{i,j}$ non-zero field elements.
Those are local Frobenius algebras that are not always symmetric, see for example \cite{Op}.
One can construct non-weakly Gorenstein algebra from such algebras that are not necessarily Liu-Schulz algebras by finding modules $M$ with $\Ext^{l}(M,M) \neq 0$ and $\Ext^{l+i}(M,M)=0$ for $i \geq 1$ as we did in \ref{theoreminfinitefindomdim}.
We thus formulate some questions more generally for quantum complete intersection algebras. We call an indecomposable module $M$ $l$-special for some $l \geq 1$ if $\Ext^l(M,M) \neq 0$ and $\Ext^i(M,M)=0$ for all $i \geq l+1$.
\begin{enumerate}
\item Can one construct $l$-special modules over symmetric algebras over an arbitrary field? Can one construct non-weakly Gorenstein algebras over an arbitrary field? Note that in our construction we had to assume that there are field elements that are not roots of unity so our construction did not work for example for finite fields. 
\item Are there non-projective modules $M$ with $\Ext^{1}(M,M)=0$ over quantum complete intersection algebras? We remark that it seems to be an open problem in general to find non-projective modules $M$ over local Frobenius algebras with $\Ext^{1}(M,M)=0$. We refer to \cite{Mar3} for more on this.
\item Is a left weakly Gorenstein algebras also right weakly Gorenstein? We remark that at the end of the article \cite{RZ} it was shown that a positive answer to this question would imply the truth of the Gorenstein symmetry conjecture.
\end{enumerate}

\end{document}